\documentclass[11pt,a4paper,twoside]{article}
\usepackage{amsthm,amsfonts,amsmath}

\newtheorem{Theorem}{Theorem}[section]

\newtheorem{Definition}{Definition}[section]
\newtheorem{Example}{Example}[section]
\newtheorem{Lemma}{Lemma}[section]
\newtheorem{Proposition}{Proposition}[section]
\newtheorem{Remark}{Remark}[section]

\author{Vladimir Rovenski\footnote{Department of Mathematics, University of Haifa, 3498838 Haifa, Israel
       \newline e-mail: \texttt{vrovenski@univ.haifa.ac.il} } }

\title{On the splitting of weak nearly cosymplectic manifolds}

\begin{document}

\date{}

\maketitle

\begin{abstract}
Weak almost contact manifolds, i.e., the~linear complex structure on the contact distribu\-tion is approximated by a nonsingular skew-symmetric tensor,
defined by the author and R.\,Wolak (2022), allowed a new look at the theory of contact manifolds.
This~article studies the curvature and topology of new structures of this type, called the weak nearly cosymplec\-tic structure and weak nearly K\"{a}hler structure. We find conditions under which weak nearly cosymplectic mani\-folds become Riemannian products
and characterize 5-dimensional weak nearly cosymplec\-tic manifolds.
Our~theorems generalize results by H.\,Endo (2005) and A.\,Nicola--G.\,Dileo--I.\,Yudin (2018)
to the context of weak almost contact geometry.


\end{abstract}


\section{Introduction}
\label{sec:00-ns}

An important class of almost contact metric manifolds $M^{\,2n+1}(\varphi,\xi,\eta,g)$ is given by cosymplectic mani\-folds, i.e., $\nabla\varphi=0$, see \cite{blair2010riemannian}.
Any such manifold is locally the product of
a real~line and a K\"{a}hler manifold $\bar M^{\,2n}(J,\bar g)$, where $J^2 = -{\rm id}$ and $\bar\nabla J=0$.
A.~Gray defined in \cite{G-70} a nearly K\"{a}hler structure $(J,\bar g)$ using condition that the symmetric part of $\bar\nabla J$ vanishes.
D.~Blair and D.~Showers defined in \cite{blair1974} a nearly cosymplectic structure $(\varphi,\xi,\eta,g)$ using a similar condition
that only the symmetric part of $\nabla\varphi$ vanishes.
The curvature and topology of nearly cosymplectic manifolds have been studied by many authors,
e.g.,~\cite{blair2010riemannian,C-MD-2016,JKK-94,K-82,LV-D-2015,NDY-2018,ru-2023}.
These odd-dimensional counterparts of nearly K\"{a}hler manifolds are useful in classifying almost contact metric manifolds~\cite{CG-1990}.
A~nearly cosymplectic structure, identified with a section of a twistor bundle, defines a harmonic map~\cite{LV-D-2015}.
In dimensions greater than 5, a nearly cosymplectic manifold is locally isometric to the Riemannian product ${\mathbb R}\times\bar M^{2n}$
or $B^5\times\bar M^{2n-4}$, where $\bar M$ is a nearly K\"{a}hler manifold and $B$ is a nearly cosymplectic mani\-fold, see~\cite{C-MD-2016}.
Any 5-dimensional nearly cosymplectic manifold has Einstein metric of positive scalar curvature, see~\cite{C-MD-2016};
and a 3-dimensional nearly cosymplectic manifold is cosymplectic, see \cite{JKK-94}.
For example, the sphere $S^5$ is endowed with nearly cosymplectic structure induced by the almost Hermitian structure of~$S^6$.

In \cite{RP-2,RWo-2,rov-117}, we introduced metric structures on a smooth manifold that genera\-lize the almost contact,
cosymplectic, Sasakian, etc. metric structures.
These so-called ``weak" structures (the~linear complex structure on the contact distribution is approximated by a nonsingular skew-symmetric tensor)
made it possible to take a new look at the classical structures and find new applications.
In~\cite{rov-2023}, we defined new structures of this type, called the {weak nearly cosymplectic structure} and {weak nearly K\"{a}hler structure}, and asked the question: \textit{under what conditions are weak nearly cosymplectic manifolds locally the Riemannian products}?

In this article, we study the differentail geometry and topology of weak
almost contact metric manifolds
and find conditions \eqref{E-nS-10} and \eqref{E-nS-04c} that are satisfied by almost contact metric manifolds and
under which weak almost cosymplectic manifolds are locally Riemannian products.
In~Section~\ref{sec:01-ns}, following the introductory Section~\ref{sec:00-ns}, we recall necessary results on weak almost contact structures.
Section~\ref{sec:02-ns} formulates auxiliary lemmas on the geometry of weak almost cosymplectic and weak almost K\"{a}hler manifolds.
In~Section~\ref{sec:04-ns}, we generalize some results of the work~\cite{NDY-2018} and prove the splitting theorem
that a weak nearly cosymplectic manifold $M^{\,2n+1}\ (n>2)$ is locally the Riemannian product of either the real line and a weak nearly K\"{a}hler manifold, or, under certain conditions, a weak nearly K\"{a}hler manifold $\bar M^{\,2n-4}(\bar\varphi,\bar g)$ with the property $\bar\nabla(\bar\varphi^{\,2})=0$ and a weak nearly cosymplectic manifold of dimension~5.
In~Section~\ref{sec:app}, using the approach of \cite{E-2005} we prove auxiliary lemmas.
Our~proofs use the properties of new tensors, as well as classical constructions.

\section{Preliminaries}
\label{sec:01-ns}

A~\textit{weak almost contact structure} on a smooth manifold $M^{\,2n+1}\ (n\ge1)$ is a set $(\varphi,Q,\xi,\eta)$,
where $\varphi$ is a $(1,1)$-tensor, $\xi$ is a vector field (called Reeb vector field), $\eta$ is a
1-form and $Q$ is a nonsingular $(1,1)$-tensor on $TM$, satisfying, see \cite{RP-2,RWo-2},
\begin{equation}\label{E-nS-2.1}
 \varphi^2 = -Q + \eta\otimes \xi, \quad \eta(\xi)=1,\quad Q\,\xi=\xi .
\end{equation}
A ``small" (1,1)-tensor $\widetilde Q= Q-{\rm id}$ is a measure of the difference between a weak almost contact structure and an almost contact one.
By \eqref{E-nS-2.1}, $\ker\eta$ is a $2n$-dimensional distribution, which we assume to be $\varphi$-invariant (as in the classical theory \cite{blair2010riemannian}, where $Q={\rm id}$).
By this and \eqref{E-nS-2.1}, $\ker\eta$ is $Q$-invariant
and the following equalities are true:
\begin{align*}
 & \varphi\,\xi=0,\quad \eta\circ\varphi=0,\quad \eta\circ Q=\eta,\quad [Q,\,\varphi]:={Q}\circ\varphi - \varphi\circ{Q}=0, \\
 & [\widetilde{Q},\varphi]:=\widetilde{Q}\circ\varphi - \varphi\circ\widetilde{Q}=0,\quad \eta\circ\widetilde Q=0,\quad \widetilde{Q}\,\xi=0 .
\end{align*}
A weak almost contact structure $(\varphi,Q,\xi,\eta)$ on a manifold $M$ will be called {\it normal} if
the tensor
${N}^{\,(1)}(X,Y) = [\varphi,\varphi](X,Y) + 2\,d\eta(X,Y)\,\xi$
is identically zero.
Here,
$[\varphi,\varphi](X,Y) = \varphi^2 [X,Y] + [\varphi X, \varphi Y] - \varphi[\varphi X,Y] - \varphi[X,\varphi Y]$
is the Nijenhuis torsion of $\varphi$
and $d\eta(X,Y) = \frac12\,\{X(\eta(Y)) - Y(\eta(X)) - \eta([X,Y])\}$
is the exterior derivative of $\eta$,
see, for example, \cite{blair2010riemannian}.
 If there is a Riemannian metric $g$ on $M$ such that
\begin{align}\label{E-nS-2.2}
g(\varphi X,\varphi Y)= g(X,Q\,Y) -\eta(X)\,\eta(Y),\quad X,Y\in\mathfrak{X}_M,
\end{align}
then $(\varphi,Q,\xi,\eta,g)$ is called a {\it weak almost contact metric structure}.
A weak almost contact manifold $M^{\,2n+1}(\varphi,Q,\xi,\eta)$ endowed with a compatible Riemannian metric $g$ is called a \textit{weak almost contact metric manifold} and is denoted by $M^{\,2n+1}(\varphi,Q,\xi,\eta,g)$.
By \eqref{E-nS-2.2},
$\eta(X)=g(X,\xi)$ and
 $g(X,Q\,X)=g(\varphi X,\varphi X)>0$
are true ;
thus, the tensor $Q$ is symmetric and positive definite.

A~1-form $\eta$ on a smooth manifold $M^{\,2n+1}$ is
\textit{contact} if $\eta\wedge (d\eta)^n\ne0$, e.g.,~\cite{blair2010riemannian}.
A~\textit{weak contact metric structure} is a weak almost contact metric structure satisfying $d\eta = \Phi$,
where the {fundamental $2$-form} $\Phi$ is defined by
$\Phi({X},{Y})=g({X},\varphi {Y}),\ {X},{Y}\in\mathfrak{X}_M$.

\begin{Lemma}\label{L-contact}
For a weak contact metric manifold $M^{\,2n+1}(\varphi,Q,\xi,\eta,g)$, the 1-form $\eta$ is contact.
\end{Lemma}

\begin{proof}
Let $e_1\in(\ker\eta)_x$ be a unit eigenvector of the self-adjoint operator $Q$ with the real eigenvalue $\lambda_1$ at a point $x\in M$.
Then $\varphi\,e_1\in(\ker\eta)_x$ is orthogonal to $e_1$
and $Q(\varphi\,e_1) = \varphi(Q\,e_1) = \lambda_1\varphi\,e_1$.
Thus, the subspace orthogonal to the plane $span\{e_1,\varphi\,e_1\}$ is $Q$-invariant.
 There exists a unit vector $e_2\in(\ker\eta)_x$ such that $e_2\perp span\{e_1,\varphi\,e_1\}$
and $Q\,e_2= \lambda_2 e_2$  for some real $\lambda_2$.
Obviously, $Q(\varphi\,e_2) = \varphi(Q\,e_2) = \lambda_2\varphi\,e_2$.
All five vectors $\{\xi, e_1, \varphi\,e_1,e_2, \varphi\,e_2\}$ are nonzero and mutually orthogo\-nal.
Continuing in the same manner, we find an orthogonal basis $\{\xi, e_1, \varphi\,e_1,\ldots, e_n, \varphi\,e_n\}$ of $T_x M$.
Since $d\eta=\Phi$, we get
 $\eta\wedge (d\eta)^n(\xi, e_1, \varphi\,e_1,\ldots, e_n, \varphi\,e_n)
 =(d\eta)^n(e_1, \varphi\,e_1,\ldots, e_n, \varphi\,e_n) \ne0$,
i.e., $\eta$ is a contact 1-form.
\end{proof}

\begin{Definition}[\cite{rov-2023}]\rm
A weak almost contact metric structure is said to be \textit{weak almost cosymplectic}, if $d\Phi=d\eta=0$.
A normal weak almost cosymplectic structure is called \textit{weak cosymplectic}.
A weak almost contact metric structure is called \textit{weak nearly cosymplectic} if
\begin{equation}\label{E-nS-00b}
 (\nabla_Y\,\varphi)Z + (\nabla_Z\,\varphi)Y = 0,\quad Y,Z\in\mathfrak{X}_M.
\end{equation}
A Riemannian manifold $(\bar M^{\,2n}, \bar g)$ of even dimension equipped with a skew-symmetric {\rm (1,1)}-tensor $\bar\varphi$
such that the tensor $\bar\varphi^{\,2}$ is negative definite will be called
a \textit{weak nearly K\"{a}hler manifold},
if $(\bar\nabla_X\,\bar\varphi)X=0\ (X\in T\bar M)$, where $\bar\nabla$ is the Levi-Civita connection of $\bar g$,~or,
\begin{equation*}
 (\bar\nabla_X\,\bar\varphi)Y + (\bar\nabla_Y\,\bar\varphi)X=0,\quad X,Y\in \mathfrak{X}_{\bar M}.
\end{equation*}
Moreover, if  $\bar\nabla\,\bar\varphi=0$ is true, then $\bar M^{\,2n}(\bar\varphi, \bar g)$ will be called a \textit{weak K\"{a}hler manifold}.
\end{Definition}


From the equalities $(\nabla_\xi\,\varphi)\,\xi = 0$ and $\varphi\,\xi=0$
we find that $\xi$ is a geodesic vector field ($\nabla_\xi\,\xi=0$).
Recall \cite{rov-2023}~that
on a weak nearly cosymplectic manifold $M^{\,2n+1}(\varphi,Q,\xi,\eta,g)$ with the property
\begin{equation}\label{E-nS-10}
 (\nabla_X\,Q)Y=0,\quad X\in\mathfrak{X}_M,\ Y\in\ker\eta,
\end{equation}
the vector field $\xi$ is Killing (the Lie derivative $\pounds_\xi\,g=0$),
and using $\nabla_\xi\,\xi=0$, we get $\nabla_\xi\,Q=0$.
Note that if we extend \eqref{E-nS-10} for $Y=\xi$, then either $\nabla\xi=0$ or $\widetilde Q=0$:
\[
 0 = (\nabla_X\,Q)\xi=\nabla_X\,\xi - Q(\nabla_X\,\xi) = -\widetilde Q(\nabla_X\,\xi).
\]

\begin{Proposition}
A three-dimensional weak nearly cosymplectic structure $(\varphi,Q,\xi,\eta,g)$ satisfy\-ing \eqref{E-nS-10}
reduces to cosymplectic one.
\end{Proposition}

\begin{proof}
By \eqref{E-nS-2.1}, the symmetric tensor $Q$ has on the plane field $\ker\eta$ the form $\lambda\,{\rm id}_{\,\ker\eta}$ for some positive $\lambda\in\mathbb{R}$.
It was shown in \cite{rov-2023} that this structure reduces to the nearly cosymplectic structure $(\tilde\varphi,\xi,\eta,\tilde g)$, where
 $\tilde\varphi = \lambda^{\,-\frac12}\,\varphi$,\
 $\tilde  g|_{\,\ker\eta} = \lambda^{\,\frac12}\,g|_{\,\ker\eta}$,\
 $\tilde g(\xi,\,\cdot) = {g}(\xi,\,\cdot)$.
The 3-dimensional nearly cosymplectic structure $(\tilde\varphi,\xi,\eta,\tilde g)$ is cosymplectic, see \cite[Theorem~4]{JKK-94}.
\end{proof}

\begin{Example}[\cite{rov-2023}]\rm
Let $\bar M(\bar\varphi, \bar g)$ be a weak nearly K\"{a}hler manifold, i.e., $(\bar\nabla_X\,\bar\varphi)X=0$. 
To~build a weak nearly cosymplectic structure on the product $M=\bar M\times\mathbb{R}$
of $\bar M$ and a line $(\mathbb{R},\partial_t)$, take any point $(x, t)$ of $M$ and set
 $\xi = (0, \partial_t),\
 \eta =(0, dt),\
 \varphi(X, \partial_t) = (\bar\varphi X, 0),\
 Q(X, \partial_t) = (-\bar\varphi^{\,2} X, \partial_t)$,
where $X\in T_x\bar M$.
Note that if $\bar\nabla_X\,\bar\varphi^2=0$ for all $X\in T\bar M$, then \eqref{E-nS-10} is true.
\end{Example}

\begin{Remark}\rm
Any weak K\"{a}hler manifold is weak nearly K\"{a}hler.
Several authors studied the problem of finding skew-symmetric parallel 2-tensors (different from almost complex structures)
on a Riemannian manifold and classified them, e.g., \cite{H-2022}.
The~idea of considering the entire bundle of almost-complex structures compatible with a given metric led to the twistor construction
and then to twistor string theory.
Thus, it may be interesting to consider the entire bundle of weak (nearly) K\"{a}hler structures that are compatible with a given~metric.
\end{Remark}

For a Riemannian manifold $(M,g)$ equipped with a~Killing vector field ${\xi}$, we get, see \cite{YK-1985},
\begin{equation}\label{E-nS-04}
 \nabla_X\nabla_Y\,{\xi} - \nabla_{\nabla_X Y}\,{\xi} = R_{\,X,\,{\xi}}\,Y ,
\end{equation}
where $R_{{X},{Y}}Z=\nabla_X\nabla_Y Z -\nabla_Y\nabla_X Z -\nabla_{[X,Y]} Z$ is the curvature tensor, e.g., \cite{KN-69}.

The curvature tensor of nearly cosymplectic manifolds satisfies $g(R_{\,\xi, Z}\,\varphi X, \varphi Y) = 0$, see \cite{E-2005};
thus the contact distribution of nearly cosymplectic manifolds is \textit{curvature invariant}:
\begin{equation}\label{E-nS-04cc}
 R_{X,Y}Z\in\ker\eta,\quad X,Y,Z\in\ker\eta.
\end{equation}
For example, any 1-form $\eta$ on a real space form has the property \eqref{E-nS-04cc}.
We will prove, see
Lemma~\ref{L-R01}, that a weak nearly cosymplectic manifold satisfies \eqref{E-nS-04cc} if we assume a weaker condition
\begin{equation}\label{E-nS-04c}
 R_{\widetilde Q X,Y}Z\in\ker\eta,\quad X,Y,Z\in\ker\eta.
\end{equation}
Taking derivative of $g(\varphi V,Z)=-g(V,\varphi Z)$, we see that $\nabla_{Y}\varphi$ of a weak nearly cosymplectic manifold is skew-symmetric:
 $g((\nabla_{Y}\varphi) V, Z)=-g((\nabla_{Y}\varphi) Z, V)$.
Taking derivative of this,
we see that $\nabla^2_{X,Y}\varphi$ is skew-symmetric:
 $g((\nabla^2_{X,Y}\varphi)V, Z)=-g((\nabla^2_{X,Y}\varphi)Z, V)$.
Recall the Ricci identity
\begin{equation}\label{E-nS-05}
 g((\nabla^2_{X,Y}\varphi)V, Z) - g((\nabla^2_{Y,X}\varphi)V, Z) = g(R_{X,Y}\varphi V, Z) + g(R_{X,Y}V, \varphi Z),
\end{equation}
see \cite{E-2005},
where the second covariant derivative operator is given by $\nabla^2_{X,Y} = \nabla_{X}\nabla_{Y} - \nabla_{\nabla_{X}Y}$.

\section{Auxiliary lemmas}
\label{sec:02-ns}

In this section, we consider a weak nearly cosymplectic manifold $M^{\,2n+1}(\varphi,Q,\xi,\eta,g)$
with conditions \eqref{E-nS-10} and \eqref{E-nS-04c} and generalize some well known results on nearly cosymplectic manifolds.

We define a (1,1)-tensor field $h$ on $M$ as in the classical case, e.g., \cite{E-2005},
\begin{equation}\label{E-c-01}
  h = \nabla\xi.
\end{equation}
Note that $h=0$ if and only if
$\ker\eta$ is integrable, i.e., $[X,Y]\in\ker\eta\ (X,Y\in\ker\eta)$.
Since $\xi$ is a geodesic vector field ($\nabla_\xi\,\xi=0$), we get $h\,\xi=0$ and $h(\ker\eta)\subset\ker\eta$.
Since $\xi$ is a Killing vector field, the tensor $h$ is skew-symmetric:
 $g(h X,\, X) = g(\nabla_X\,\xi, X) = \frac12\,(\pounds_\xi\,g)(X,X) = 0$.
We~also get $\eta\circ h = 0$ and
  $d\,\eta(X,\,\cdot) = \nabla_X\,\eta = g(hX,\,\cdot)$.
The following lemma generalizes Lemma~3.1 in \cite{E-2005}.

\begin{Lemma}\label{L-nS-02}
For a weak nearly cosymplectic manifold $M^{\,2n+1}(\varphi,Q,\xi,\eta,g)$ we obtain
\begin{eqnarray}
\label{E-nS-01a}
 && (\nabla_X\,h)\,\xi = -h^2 X,\\
\label{E-nS-01c}
 && (\nabla_X\,\varphi)\,\xi = -\varphi\,h X .
\end{eqnarray}
Moreover, if the condition \eqref{E-nS-10} is true, then
\begin{eqnarray}
\label{E-nS-01b}
 && h\,\varphi + \varphi\,h =0\quad (h\ {\rm anticommutes\ with}\ \varphi),\\
\label{E-nS-01d}
 && h\,Q = Q\,h\quad (h\ {\rm commutes\ with}\ Q).
\end{eqnarray}
\end{Lemma}

\begin{proof}
Differentiating the equality $h\,\xi=0$ and using
\eqref{E-c-01}, we obtain \eqref{E-nS-01a}:
\[
 0=\nabla_X\,(h\,\xi) = (\nabla_X\,h)\,\xi +h(\nabla_X\,\xi) =(\nabla_X\,h)\,\xi +h^2 X.
\]
Differentiating the equality $g(\varphi\,Y,\xi)=0$ yields
 $0=X g(\varphi\,Y,\xi) = g((\nabla_X\,\varphi)Y, \xi) +g(\varphi\,Y, h X)$.
Summing this with the equality $g((\nabla_Y\,\varphi)X, \xi) +g(\varphi\,X, h Y)=0$ and applying \eqref{E-nS-00b}, gives \eqref{E-nS-01b}:
\[
 0 = g(\varphi\,Y, h X) + g(\varphi\,X, h Y) = - g( (h\,\varphi + \varphi\,h)X , Y).
\]
Using $\varphi\,\xi=0$ and the definition \eqref{E-c-01}, we get \eqref{E-nS-01c}:
 $(\nabla_X\,\varphi)\,\xi = -\varphi(\nabla_X\,\xi)= -\varphi\,h X$.
By \eqref{E-nS-01b} and \eqref{E-nS-2.1}, using the equalities $h\,\xi=0$ and $\eta\circ h=0$, we obtain \eqref{E-nS-01d}.
\end{proof}

The following four lemmas generalize certain formulas in Lemmas~3.2 -- 3.5 in \cite{E-2005}.

\begin{Lemma}\label{L-R02}
For a weak nearly cosymplectic manifold satisfying \eqref{E-nS-10} we obtain
\begin{eqnarray}\label{E-3.29}
 &&\hskip-7mm
 g((\nabla_X\,\varphi)\varphi Y, Z) = g((\nabla_X\varphi)Y, \varphi Z) + \eta(Y) g(hX, Z) +\eta(Z) g(h X, Q Y) ,\\
\label{E-3.30}
 &&\hskip-7mm
 g((\nabla_{\varphi X}\varphi)Y, Z) = g((\nabla_X \varphi)Y, \varphi Z) + \eta(X) g(hZ, Y) + \eta(Z) g(h X, Q Y) ,\\
\label{E-3.31}
 &&\hskip-13mm
 g((\nabla_{\varphi X} \varphi)\varphi Y, Z) = g((\nabla_X \varphi)Q Z, Y) {+} \eta(X) g(hZ, \varphi Y) {+} \eta(Y) g(hX, \varphi Z)
  {-}\eta(Z) g(\varphi h X, \widetilde Q Y) .
\end{eqnarray}
\end{Lemma}

\begin{proof}
As in the proof of \cite[Lemma~3.4]{E-2005},
differentiating \eqref{E-nS-2.2} and using \eqref{E-nS-01b}, \eqref{E-nS-10} and the skew-symmetry of $\nabla_X\,\varphi$, we get \eqref{E-3.29}.
We obtain \eqref{E-3.30} from \eqref{E-3.29} by the condition~\eqref{E-nS-00b}.
Replacing $Y$ by $\varphi Y$ in \eqref{E-3.30} and using \eqref{E-3.29} and \eqref{E-nS-2.1}, we get \eqref{E-3.31}.
\end{proof}

\begin{Lemma}\label{L-R01}
The curvature tensor of a weak nearly cosymplectic manifold satisfies the equality
\begin{eqnarray}\label{E-3.4}
 g(R_{\varphi X,Y}Z, V) +g(R_{X,\varphi Y}Z, V) +g(R_{X,Y}\varphi Z, V) +g(R_{X,Y}Z, \varphi V) = 0.
\end{eqnarray}
Moreover, if the conditions \eqref{E-nS-10} and \eqref{E-nS-04c} are true, then
\begin{eqnarray}
\label{E-nS-04ccc}
 && g(R_{\,\xi, Z}\,\varphi X,\varphi Y) = 0, \\
\label{E-3.6}
&& g(R_{\varphi X,\varphi Y} Z, V) = g(R_{X, Y}\varphi Z, \varphi V) - \frac12\,\delta(X,Y,Z,V), \\
\label{E-3.5}
\nonumber
&& g(R_{\varphi X,\varphi Y}\varphi Z, \varphi V) = g(R_{Q X, Q Y}Z, V) -\eta(X)\,g(R_{\,\xi,Q Y}Z, V) \\
&&\quad  +\,\eta(Y)\,g(R_{\,\xi, Q X}Z, V) + \frac12\,\delta(\varphi X,\varphi Y,Z,V),
\end{eqnarray}
where
 $\delta(X,Y,Z,V) = g(R_{X, Y}\widetilde Q Z, V) +g(R_{X, Y}Z, \widetilde Q V) -g(R_{\widetilde Q X, Y}Z, V) -g(R_{X, \widetilde Q Y}Z, V)$.
\end{Lemma}

\begin{Lemma}\label{L-R03}
For a weak nearly cosymplectic manifold satisfying \eqref{E-nS-10} and \eqref{E-nS-04c}, we obtain
\begin{eqnarray}\label{E-3.50}
 g((\nabla_{X}\,\varphi)Y, \varphi h Z) = \eta(X)\,g(hY, hQZ) -\eta(Y)\,g(hX, hQZ).
\end{eqnarray}
\end{Lemma}

\begin{Lemma}\label{L-nS-04}
For a weak nearly cosymplectic manifold $M^{\,2n+1}(\varphi,Q,\xi,\eta,g)$ satisfying
\eqref{E-nS-04c}, we get
\begin{eqnarray}\label{E-3.24}
 && (\nabla_X\,h)Y = g(h^2 X, Y)\,\xi - \eta(Y)\,h^2 X , \\
\label{E-3.23}
 && R_{\,\xi, X}Y = -(\nabla_X\,h)Y, \\
\label{E-3.25}
 && {\rm Ric}\,(\xi, Z) = -\eta(Z)\,{\rm tr}\,h^2 .
\end{eqnarray}
In particular, $\nabla_\xi\,h=0$ and ${\rm tr}(h^2) = const$.
By \eqref{E-3.24}--\eqref{E-3.23}, we get
\begin{equation}\label{E-3.23b}
  g(R_{\,\xi, X}Y, Z) = -g((\nabla_X\,h)Y, Z) = \eta(Y)\,g(h^2 X, Z) -\eta(Z)\,g(h^2 X, Y).
\end{equation}
\end{Lemma}

\begin{proof}
By \eqref{E-nS-04} (since $\xi$ is a Killing vector) and \eqref{E-c-01}, we get \eqref{E-3.23}.
Replacing $Y$ by $\varphi Y$ and $Z$ by $\varphi Z$ in $g(R_{\,\xi, X}Y, Z) = -g((\nabla_X\,h)Y, Z)$, see \eqref{E-3.23},
and using \eqref{E-nS-04ccc}, we get $g((\nabla_X\,h)\,\varphi Y, \varphi Z) = 0$,~hence,
\begin{equation}\label{E-nS-03f}
 g((\nabla_X\,h) Y, Z) = 0,\quad Y, Z\in\ker\eta.
\end{equation}
Then, using \eqref{E-nS-03f}, we find the $\xi$-component and $\ker\eta$-component of $(\nabla_X\,h)Y$:
\begin{eqnarray*}
 && g((\nabla_X\,h)Y,\,\xi) = g(\nabla_X(h\,Y), \,\xi) = -g(h\,Y, \,\nabla_X\,\xi) = g(h^2 X, Y),\\
 && g((\nabla_X\,h)Y, Z) = \eta(Y)\,g((\nabla_X\,h)\,\xi, \,Z)
 = -\eta(Y)\,g(h^2 X, Z)\quad (Z\in\ker\eta),
\end{eqnarray*}
from which \eqref{E-3.24} follows.
From \eqref{E-3.24} with $X=\xi$ we find $\nabla_\xi\,h=0$.

Let $\{e_i\}\ (i=1,\ldots,2n+1)$ be a local orthonormal frame on $M$ with $e_{2n+1}=\xi$.
Putting $X=Y =e_i$ in \eqref{E-3.24}, then using \eqref{E-nS-04cc} and summing over $i=1,\ldots,2n+1$, we get~\eqref{E-3.25}.
Replacing $Y$ by $h Y$ in \eqref{E-3.24}, putting $Y=e_i$ in the gotten equation and summing over $i=1,\ldots, 2n+1$,
we get ${\rm tr}\,((\nabla_X\,h)\,h) = 0$. This implies $X({\rm tr}(h^2)) = 0\ (X\in\mathfrak{X}_M)$, i.e., ${\rm tr}(h^2) = const$.
\end{proof}

\begin{Remark}\label{Rem-delta}\rm
The function $\delta$ of a weak nearly cosymplectic manifold has the following symmetries:
\begin{eqnarray*}
 \delta(Y,X, Z,V) = \delta(X,Y, V,Z) = \delta(Z,V, X,Y) = -\delta(X,Y, Z,V) .
\end{eqnarray*}
If \eqref{E-nS-04c} is true, then by \eqref{E-3.23b}, we get
$\delta(\xi,Y,Z,V)\!=\delta(X,\xi,Z, V)\!=\delta(X,Y,\xi, V)\!=\delta(X,Y,Z, \xi)\!=0$.
\end{Remark}

\section{Main results}
\label{sec:04-ns}

In Section~\ref{subsec:04a-ns}, we prove the splitting of weak nearly cosymplectic manifolds with conditions \eqref{E-nS-10} and~\eqref{E-nS-04c}.
For almost contact metric manifolds, conditions \eqref{E-nS-10} and \eqref{E-nS-04c} become trivial.
In Section~\ref{subsec:04-ns} we characterize 5-dimensional weak nearly cosymplectic manifolds.

\subsection{The splitting theorem}
\label{subsec:04a-ns}

The following proposition generalizes \cite[Proposition~4.2]{NDY-2018}.

\begin{Proposition}\label{Prop-4.1}
For a weak nearly cosymplectic manifold with conditions \eqref{E-nS-10} and \eqref{E-nS-04c},
the eigenvalues (and their multiplicities) of the symmetric operator $h^2$ are constant.
\end{Proposition}

\begin{proof}
From \eqref{E-3.23b} and Lemma~\ref{L-nS-04} we obtain
\begin{equation}\label{E-nS-11}
 (\nabla_X\, h^2)Y
 = h(\nabla_X\, h)Y + (\nabla_X\, h) h Y
 = g(X, h^3Y)\,\xi - \eta(Y)\,h^3 X.
\end{equation}
Consider an eigenvalue $\mu$ of $h^2$ and a local unit vector field $Y$ orthogonal to $\xi$ such that $h^2Y = \mu Y$.
Applying \eqref{E-nS-11} for any nonzero vector fields $X$ and $Y\perp\xi$, we find $g((\nabla_X\, h^2)Y, Y)=0$, thus
\begin{eqnarray*}
 && 0 = g((\nabla_X\, h^2)Y, Y) = g(\nabla_X\,(h^2Y), Y) - g(h^2(\nabla_X\,Y), Y ) \\
 &&\ \ = X(\mu)\,g(Y, Y) + \mu\,g(\nabla_X\,Y, Y) - g(\nabla_X\,Y, h^2Y) = X(\mu)\,g(Y, Y),
\end{eqnarray*}
which implies that $X(\mu) = 0$ for all $X\in\mathfrak{X}_M$.
\end{proof}

By Proposition~\ref{Prop-4.1}, the spectrum of the self-adjoint operator $h^2$ has the~form
\begin{equation}\label{E-nS-11b}
  Spec(h^2) = \{0, -\lambda_1^2,\ldots -\lambda_r^2\} ,
\end{equation}
where $\lambda_i$ is a positive real number and $\lambda_i\ne \lambda_j$ for $i\ne j$.
If $X\ne0$ is an eigenvector of $h^2$ with eigenvalue $-\lambda^2_i$, then $X, \varphi X, hX$ and $h\,\varphi X$ are orthogonal
nonzero eigenvectors of $h^2$ with eigenvalue $-\lambda^2_i$.
Since $h(\xi)=0$, the eigenvalue 0 has multiplicity $2p+1$ for some integer $p\ge0$.

Denote by $D_0$ the smooth distribution of the eigenvectors with eigenvalue 0 orthogonal to $\xi$.
Let $D_i$ be the smooth distribution of the eigenvectors with eigenvalue $-\lambda^2_i$.
Remark that the distributions $D_0$ and $D_i$ belong to $\ker\eta$ and are $\varphi$-invariant and $h$-invariant.

The following proposition generalizes \cite[Proposition~4.3]{NDY-2018}.

\begin{Proposition}\label{P-4.3}
Let $M^{\,2n+1}(\varphi,Q,\xi,\eta,g)$ be a weak nearly cosymplectic manifold with conditions \eqref{E-nS-10} and \eqref{E-nS-04c},
and let the spectrum of the self-adjoint operator $h^2$ have the form \eqref{E-nS-11b}. Then,

$(a)$ each distribution $[\xi]\oplus D_i\ (i = 1,\ldots, r)$ is integrable with totally geodesic leaves.

\noindent
Moreover, if the eigenvalue $0$ of $h^2$ is not simple, then

$(b)$ the distribution $D_0$ is integrable with totally geodesic leaves, and each leaf of $D_0$ is endowed with a weak nearly K\"{a}hler structure
$(\bar\varphi, \bar g)$
with the property $\bar\nabla(\bar\varphi^{\,2})=0$;

$(c)$ the distribution $[\xi]\oplus D_1\oplus\ldots\oplus D_r$ is integrable with totally geodesic leaves.
\end{Proposition}

\begin{proof}
Consider an eigenvector $X$ of $h^2$ with eigenvalue $-\lambda^2_i$. Then $\nabla_X\,\xi = hX \in D_i$.
On the other hand, \eqref{E-nS-11} implies that $\nabla_\xi\,h^2 = 0$, and thus $\nabla_\xi X$ is also an eigenvector of $h^2$ with eigenvalue $-\lambda^2_i$.
Now, taking $X, Y \in D_i$ and applying \eqref{E-nS-11}, we get
\[
 h^2(\nabla_X\,Y) = -\lambda^2_i\,\nabla_X Y - (\nabla_X\,h^2)Y = -\lambda^2_i\,\nabla_X Y + \lambda^2_i\,g(X, hY)\,\xi.
\]
Therefore,
 $h^2(\varphi^2\nabla_X Y) = \varphi^2(h^2\nabla_X Y) = -\lambda^2_i\,\varphi^2(\nabla_X Y)$;
hence, $\varphi^2\nabla_X Y \in D_i$.
Similarly, using \eqref{E-nS-01d}, we get $\widetilde Q\nabla_X Y \in D_i$.
It follows that
\[
 \nabla_X Y = -\widetilde Q\,\nabla_X Y -\varphi^2\nabla_X Y + \eta(\nabla_X Y)\,\xi,
\]
see \eqref{E-nS-2.1}, belongs to the distribution $[\xi]\oplus D_i$.
This proves (a).

Assume that the eigenvalue $0$ of $h^2$ is not simple.
By \eqref{E-nS-11}, we get $(\nabla_X\,h^2)Y = 0$ for any linear independent vectors $X,Y$ in $D_0$, hence $h^2(\nabla_X Y) = 0$.
Moreover,
\[
 g(\nabla_X Y, \xi) = -g(Y, \nabla_X\,\xi) = -g(Y, h X) = 0.
\]
Thus, the distribution $D_0$ defines a totally geodesic foliation.
By \eqref{E-nS-01b} and \eqref{E-nS-01d}, the leaves of $D_0$ are $\varphi$-invariant and $Q$-invariant.
Thus, the weak nearly cosymplectic structure on $M$ with conditions \eqref{E-nS-10}  and \eqref{E-nS-04c}
induces a weak nearly K\"{a}hler structure $(\bar\varphi, \bar g)$ on each leaf of $D_0$ with the property $\bar\nabla(\bar\varphi^{\,2})=0$,
where $\bar\nabla$ is the Levi-Civita connection of $\bar g$. This proves~(b).

To prove (c) taking (a) into account, it is enough to show that $g(\nabla_X Y, Z) = 0$ for every
$X\in D_i$, $Y\in D_j\ (i\ne j)$ and $Z \in D_0$.
Indeed, from \eqref{E-nS-11}, we obtain
\begin{eqnarray*}
 g(\nabla_X\,Y, Z) && = -(1/\lambda^2_j)\,g(\nabla_X (h^2Y), Z)
 = -(1/\lambda^2_j)\,g((\nabla_X\, h^2)Y + h^2(\nabla_X Y), Z) \\
 && = -(1/\lambda^2_j)\,\eta(Z)\,g(X, h^3Y) -(1/\lambda^2_j)\,g(\nabla_X Y, h^2 Z),
\end{eqnarray*}
which vanishes since $\eta(Z)=0$ and $h^2Z=0$.
\end{proof}

The following proposition generalizes \cite[Proposition~4.1]{NDY-2018} and does not use Lemmas~\ref{L-R02}--\ref{L-R03}.

\begin{Proposition}\label{Th-4.1}
For a weak nearly cosymplectic (non-weak-cosymplectic) manifold, $h\equiv0$ if and only if the manifold
is locally isometric to the Riemannian product of a real line and a weak nearly K\"{a}hler (non-weak-K\"{a}hler) manifold.
\end{Proposition}

\begin{proof}
For every vector fields $X, Y$ orthogonal to $\xi$ we have
\begin{equation}\label{E-c-01b}
 2\,d\eta(X, Y ) = g(\nabla_X\,\xi, Y) - g(\nabla_Y\,\xi, X) = 2\,g(hX, Y).
\end{equation}
Thus, by the condition $h = 0$, the contact distribution $\ker\eta$ is integrable.
Any integral submanifold of $\ker\eta$ is a totally geodesic hypersurface.
Indeed, we have
 $g(\nabla_X\,Y, \xi)=-g(Y, hX)=0$ for every $X,Y\in\ker\eta$.
Since $\nabla_\xi\,\xi=0$, by de Rham Decomposition Theorem (e.g., \cite{KN-69}), the manifold is locally isometric to the Riemannian product $\mathbb{R}\times\bar M$.
The weak almost contact metric structure induces on $\bar M$ a weak nearly K\"{a}hler structure.
Conversely, if a weak nearly cosymplectic manifold is locally isometric to the Riemannian product $\mathbb{R}\times\bar M$, where $\bar M$ is a
weak nearly K\"{a}hler manifold and $\xi = (0, \partial_t)$, then $d\eta(X, Y)=0\ (X,Y\in\ker\eta)$.
By \eqref{E-c-01b} and $h\,\xi=0$, we get~$h=0$.
\end{proof}

We will generalize Theorem~4.5 in \cite{NDY-2018} on splitting of nearly cosymplectic manifolds.

\begin{Theorem}
Let $M^{\,2n+1}(\varphi, Q, \xi, \eta, g)$ be a weak nearly cosymplectic (non-weak-cosymplectic) manifold of dimension $2n+1>5$
with conditions \eqref{E-nS-10} and \eqref{E-nS-04c}.
Then $M$ is locally isometric to one of the following Riemannian products:
\[
 \mathbb{R}\times \bar M^{\,2n},\quad B^5 \times\bar M^{\,2n-4},
\]
where $\bar M$ is endowed with a weak nearly K\"{a}hler structure $(\bar\varphi,\bar g)$
with the property $\bar\nabla(\bar\varphi^{\,2})=0$,
and $B^5$ is a 5-dimensional weak nearly cosymplectic (non-weak-cosymplectic) manifold satisfying~\eqref{E-nS-10} and~\eqref{E-nS-04c}.
If the manifold $M$ is complete and simply connected, then
the isometry is global.
\end{Theorem}

\begin{proof}
If $h\equiv0$, then by Proposition~\ref{Th-4.1}, $M$ is locally isometric to $\mathbb{R}\times \bar M^{\,2n}$.
Let $h\ne0$ on $\ker\eta\setminus\{0\}$ and \eqref{E-nS-11b},
where $r\ge1$ and each $\lambda_i$ is a positive number.
Since $\dim M > 5$, by Theorem~\ref{T-4.4} in Section~\ref{subsec:04-ns}, the eigenvalue 0 is not a simple eigenvalue.
By (b) and (c) of Proposition~\ref{P-4.3}, and according to de Rham Decomposition Theorem (e.g., \cite{KN-69}), $M$ is locally isometric to the Riemannian product $B\times\bar M$, where $B$ is an integral submanifold
of the distribution $[\xi]\oplus D(-\lambda_1^2)\oplus\ldots\oplus D(-\lambda_r^2)$, and $\bar M$ is an integral submanifold of $D_0$,
which is endowed with a weak nearly K\"{a}hler structure $(\bar\varphi,\bar g)$ and, by the condition \eqref{E-nS-10},
has the property $\bar\nabla(\bar\varphi^{\,2})=0$.

Note that $B$ is endowed with an induced weak nearly cosymplectic (non-weak-cosymplectic) structure, for which 0 is a simple eigenvalue of the operator~$h^2$.
By~Theorem~\ref{T-4.4} in Section~\ref{subsec:04-ns}, $B$ is a 5-dimensional manifold and $\lambda_1 =\ldots = \lambda_r$.
Consequently, $\dim\bar M=2n-4$.
If the manifold $M$ is complete and simply connected, then we apply the de Rham Decomposition Theorem.
\end{proof}

\subsection{Characterization of 5-dimensional weak nearly cosymplectic manifolds}
\label{subsec:04-ns}

Here, we use Lemmas~\ref{L-R02}--\ref{L-R03} to characterize 5-dimensional weak nearly cosymplectic manifolds.


The following result generalizes Theorem~4.4 in \cite{NDY-2018} on 5-dimensional cosymplectic manifolds.

\begin{Theorem}\label{T-4.4}
Let $M^{\,2n+1}(\varphi, Q, \xi, \eta, g)$ be a weak nearly cosymplectic manifold with conditions \eqref{E-nS-10} and \eqref{E-nS-04c}
such that $0$ is a simple eigenvalue of $h^2$. Then $M$ is a 5-dimensional manifold.
\end{Theorem}

\begin{proof}
 We consider 2-forms $\Phi_k(X,Y)=g(\varphi h^k X, Y)$, where $k = 0, 1, 2$; in particular, $\Phi_0=-\Phi$.
It is easy to calculate
  $3\,d\Phi(X,Y,Z) = g((\nabla_X\,\varphi)Z,Y) +g((\nabla_Y\,\varphi)X, Z) +g((\nabla_Z\,\varphi)Y, X)$,
see \cite{RP-2}.
We will show that
\begin{equation}\label{E-c-03}
 d\Phi_0 = 3\,\eta\wedge \Phi_1,\quad d\Phi_1 = 3\,\eta\wedge \Phi_2.
\end{equation}
Indeed, applying \eqref{E-c-01} and $\varphi\,\xi=0$, we find the $\xi$-component of $(\nabla_X\,\varphi) Y$:
\begin{equation}\label{E-c-03xi}
  g((\nabla_X\,\varphi) Y, \xi) = - g((\nabla_X\,\varphi)\,\xi, Y) = g(\varphi\nabla_X\,\xi, Y) = g(\varphi h X, Y).
\end{equation}
Replacing $Z$ by $\varphi Z$ in \eqref{E-3.50} and using \eqref{E-nS-01b}, we obtain
\begin{eqnarray}\label{E-3.50b}
 && g((\nabla_{X}\,\varphi)Y, -\varphi^2 h Z) = \eta(X)\,g(hY, h\varphi QZ) -\eta(Y)\,g(hX, h\varphi QZ) .
\end{eqnarray}
By conditions, $h\ne0$ on $\ker\eta\setminus\{0\}$, thus from \eqref{E-3.50b} we get
 $g((\nabla_{X}\,\varphi)Y, V) =0$ for $X,Y,V\in\ker\eta$.
By the above and \eqref{E-c-03xi}, using $X=X^\top+\eta(X)\,\xi$ and $Y=Y^\top+\eta(Y)\,\xi$, we obtain
\begin{eqnarray}\label{E-c-03b}
\nonumber
 && g((\nabla_{X}\,\varphi)Y, V) = \eta(V)\,g((\nabla_{X^\top}\,\varphi)Y^\top, \xi)
 +\eta(X)\,g((\nabla_{\,\xi}\,\varphi)Y^\top, V) +\eta(Y)\,g((\nabla_{X^\top}\,\varphi)\,\xi, V) \\
\nonumber
 && = -\eta(V)\,g((\nabla_{X^\top}\,\varphi)\,\xi, Y^\top)
 -\eta(X)\,g((\nabla_{Y^\top}\,\varphi)\,\xi, V) +\eta(Y)\,g((\nabla_{X^\top}\,\varphi)\,\xi, V) \\
 && = \eta(V)\,g(\varphi h X, Y) +\eta(X)\,g(\varphi h Y, V) +\eta(Y)\,g(\varphi h V, X),
\end{eqnarray}
which implies that $d\Phi_0 = 3\,\eta\wedge \Phi_1$.
Similarly, using \eqref{E-c-03b} and \eqref{E-3.23b}, we get
\begin{eqnarray*}
 g((\nabla_X\,(\varphi\,h))Y, Z) = g((\nabla_X\,\varphi)hY, Z) + g(\varphi(\nabla_X\,h)Y, Z) \\
 = \eta(X)\,g(\varphi\,h^2 Y, Z) +\eta(Y)\,g(\varphi\,h^2 Z, X) +\eta(Z)\,g(\varphi\,h^2 X, Y),
\end{eqnarray*}
which implies $d\Phi_1 = 3\,\eta\wedge \Phi_2$ and completes the proof of \eqref{E-c-03}.
From \eqref{E-c-03} we obtain
\begin{equation}\label{E-Nic-27}
 0= d^2\Phi_0 = 3\,d\eta\wedge \Phi_1 - 3\,\eta\wedge d\Phi_1 =3\,d\eta\wedge \Phi_1.
\end{equation}
 Next we will show that if 0 is a simple eigenvalue of $h^2$, then $\eta$ is a contact form.
 We assume \eqref{E-nS-11b} with $r\ge1$, 0 being a simple eigenvalue.
From \eqref{E-3.23b} with $Y=\xi$, using \eqref{E-nS-01a}, we find the $\xi$-sectional curvature:
\begin{equation}\label{E-nS-K}
 K(\xi,X) = g(hX, hX)\quad (X\in\ker\eta,\ g(X,X)=1).
\end{equation}
By \eqref{E-nS-K} and the assumption, the $\xi$-sectional curvature of $M$ is positive.
By \cite[Theorem~3]{rov-117}, we get a weak K-contact structure on $M$
(i.e., a weak contact metric manifold, whose Reeb vector field is Killing, see \cite{rov-117}); thus, $\eta$ is a contact 1-form.
If $2n + 1 > 5$, $\eta$ being a contact form (see Lemma~\ref{L-contact}), from $d\eta\wedge\Phi_1 = 0$, see \eqref{E-Nic-27},
we get $\Phi_1 = 0$ -- a contradiction to Proposition~3.2 in~\cite{NDY-2018}.
Hence, $M$ is a 5-dimensional manifold and the multiplicity of the eigenvalue $-\lambda^2$ is 4.
\end{proof}

\section{Proofs of Lemmas~\ref{L-R01} and \ref{L-R03}}
\label{sec:app}

Here, in the proofs of lemmas we use the approach of \cite{E-2005}, and our formulas also contain terms depending on the tensors $Q$ and $\widetilde Q$.

\smallskip\textbf{Proof of Lemma~\ref{L-R01}}.
The proof of \eqref{E-3.4} is similar to the proof of equation (3.4) in \cite{E-2005}, we present it
because some formulas appearing in the proof of \eqref{E-3.4} are also used in the proof of \eqref{E-3.6}.
Differentiating \eqref{E-nS-00b}, we find
\begin{eqnarray}\label{EF-nS-01}
 (\nabla^2_{X,Y}\,\varphi)Z + (\nabla^2_{X,Z}\,\varphi)Y = 0.
\end{eqnarray}
Applying the Ricci identity \eqref{E-nS-05}, from \eqref{EF-nS-01} and the skew-symmetry of $\nabla^2_{X,Y}\,\varphi$ we get
\begin{eqnarray}\label{E-3.7}
 g(R_{X,Y}Z, \varphi V) -g(R_{X,Y}V, \varphi Z)  + g((\nabla^2_{X,Z}\,\varphi) Y, V) - g((\nabla^2_{Y,Z}\,\varphi)X, V) = 0.
\end{eqnarray}
By Bianchi and Ricci identities, we find
\begin{eqnarray}\label{E-3.8}
\nonumber
 && g(R_{X,Y}Z, \varphi V) = -g(R_{Y,Z}X, \varphi V) - g(R_{Z,X}Y, \varphi V) \\
 && = g((\nabla^2_{Y,Z}\,\varphi) V, X) - g((\nabla^2_{Z,Y}\,\varphi)V, X)
   -g(R_{Y,Z}V, \varphi X) - g(R_{Z,X}Y, \varphi V).
\end{eqnarray}
Substituting \eqref{E-3.8} into \eqref{E-3.7}, it follows that
\begin{eqnarray}\label{E-3.9}
\nonumber
 && g(R_{X,Z}Y, \varphi V) -g(R_{X,Y}V, \varphi Z) - g(R_{Y,Z}V, \varphi X) \\
 && -g((\nabla^2_{Z,Y}\,\varphi)V, X) - g((\nabla^2_{X,Z}\,\varphi)V, Y) = 2\,g((\nabla^2_{Y,Z}\,\varphi)X, V).
\end{eqnarray}
On the other hand, using the Ricci identity \eqref{E-nS-05}, we see that
\begin{eqnarray}\label{E-3.10}
 g(R_{X,Z}Y, \varphi V) -g(R_{X,Z}V, \varphi Y)  - g((\nabla^2_{X,Z}\,\varphi)Y, V) +g((\nabla^2_{Z,X}\,\varphi)Y, V) = 0.
\end{eqnarray}
Adding \eqref{E-3.10} to \eqref{E-3.9}, we get
\begin{eqnarray}\label{E-3.11}
  2\,g(R_{X,Z}Y, \varphi V) -g(R_{X,Y}V, \varphi Z)  - g(R_{Y,Z}V, \varphi X) -g(R_{X,Z}V, \varphi Y)
  = 2\,g((\nabla^2_{Y,V}\,\varphi)Z, X).
\end{eqnarray}
Swapping $Y$ and $V$ in \eqref{E-3.11}, we find
\begin{eqnarray}\label{E-3.12}
  2\,g(R_{X,Z}V, \varphi Y) -g(R_{X,V}Y, \varphi Z)  - g(R_{V,Z}Y, \varphi X) -g(R_{X,Z}Y, \varphi V)
  = 2\,g((\nabla^2_{V,Y}\,\varphi)Z, X).
\end{eqnarray}
Subtracting \eqref{E-3.12} from \eqref{E-3.11}, and using the Bianchi and Ricci identities, we get
the equality,
which by replacing $Z$ and $Y$ gives \eqref{E-3.4}.
Replacing $X$ by $\varphi X$ in \eqref{E-3.4} and using \eqref{E-nS-2.1}, we have
\begin{eqnarray}\label{E-3.13}
\nonumber
 && -g(R_{Q X,Y}Z, V) +\eta(X)\,g(R_{\,\xi,Y}Z, V) +g(R_{\varphi X, \varphi Y}Z, V) \\
 && +g(R_{\varphi X, Y}\varphi Z, V) +g(R_{\varphi X, Y}Z, \varphi V) = 0.
\end{eqnarray}
Exchanging $X$ and $Y$ in \eqref{E-3.13}, we find
\begin{eqnarray}\label{E-3.14}
  g(R_{X, Q Y}Z, V) {+}\eta(Y)\,g(R_{\,\xi,X}Z, V) {-} g(R_{\varphi X, \varphi Y}Z, V)
  {+} g(R_{\varphi Y, X}\varphi Z, V) {+} g(R_{\varphi Y, X}Z, \varphi V) = 0.
\end{eqnarray}
Subtracting \eqref{E-3.14} from \eqref{E-3.13}, we obtain
\begin{eqnarray}\label{E-3.15}
\nonumber
 && 2\,g(R_{\varphi X,\varphi Y}Z, V) -2\,g(R_{X,Y}Z, V) +\eta(X)\,g(R_{\,\xi,Y}Z, V) -\eta(Y)\,g(R_{\,\xi,X}Z, V) \\
\nonumber
 && + g(R_{\varphi X, Y}\varphi Z, V) - g(R_{\varphi Y, X}\varphi Z, V) + g(R_{\varphi X, Y}Z, \varphi V) - g(R_{\varphi Y, X}Z, \varphi V) \\
 && - g(R_{\widetilde Q X, Y}Z, V) - g(R_{X, \widetilde Q Y}Z, V) = 0.
\end{eqnarray}
Then, replacing $Z$ by $\varphi Z$ and also $V$ by $\varphi V$ in \eqref{E-3.4} and using \eqref{E-nS-2.1}, we get two equations
\begin{eqnarray}\label{E-3.16}
 & g(R_{X, Y}Q Z, V) = \eta(Z)\,g(R_{X,Y}\,\xi, V) {+} g(R_{X,Y}\varphi Z, \varphi V) {+} g(R_{X,\varphi Y}\varphi Z, V) {+} g(R_{\varphi X,Y}\varphi Z, V),\\
\label{E-3.17}
 & g(R_{X, Y}Z, Q V) = \eta(V)\,g(R_{X,Y}Z,\,\xi) {+} g(R_{X,Y}\varphi Z, \varphi V) {+} g(R_{\varphi X, Y}Z, \varphi V) {+} g(R_{X,\varphi Y}Z, \varphi V).
\end{eqnarray}
Adding \eqref{E-3.16} to \eqref{E-3.17},
and substituting the gotten equation into \eqref{E-3.15}, we have
\begin{eqnarray}\label{E-3.18}
\nonumber
 && 2\,g(R_{\varphi X,\varphi Y}Z, V) -2\,g(R_{X,Y}\varphi Z, \varphi V) -\eta(Z)\,g(R_{X,Y}\,\xi, V)  -\eta(V)\,g(R_{X,Y}Z,\,\xi) \\
 && +\,\eta(X)\,g(R_{\,\xi,Y}Z, V) -\eta(Y)\,g(R_{\,\xi,X}Z, V) +\delta(X,Y,Z,V) = 0.
\end{eqnarray}
Replacing $X$ by $\varphi X$ and also $Y$ by $\varphi Y$ in \eqref{E-3.18} and using \eqref{E-nS-2.1}, we obtain
\begin{eqnarray}\label{E-3.19}
\nonumber
 && 2\,g(R_{Q X, Q Y}Z, V) {-} 2\,\eta(X)\,g(R_{\,\xi, Q Y}Z, V) {+} 2\,\eta(Y)\,g(R_{\,\xi,Q X}Z, V)
 {-} 2\,g(R_{\varphi X, \varphi Y}\,\varphi Z, \varphi V) \\
 && +\,\delta(\varphi X, \varphi Y,Z,V) = \eta(Z)\,g(R_{\,\xi, V}\varphi X,\varphi Y) -\eta(V)\,g(R_{\,\xi, Z}\,\varphi X, \varphi Y) .
\end{eqnarray}
Replacing $V$ by $\xi$ in \eqref{E-3.19}, we obtain
\begin{eqnarray}\label{E-3.20}
\nonumber
 && 2\,g(R_{\,Q X, Q Y}Z, \xi) -2\,\eta(X)\,g(R_{\,\xi, Q Y}Z, \xi) +2\,\eta(Y)\,g(R_{\,\xi,Q X}Z, \xi)  \\
 && +\,g(R_{\,\xi, Z}\,\varphi X, \varphi Y)  +\delta(\varphi X, \varphi Y,Z, \,\xi) = 0.
\end{eqnarray}
Replacing $X$ by $\varphi X$ and $Y$ by $\varphi Y$ in \eqref{E-3.20},
and adding the result to \eqref{E-3.20},
we get
gives
\begin{eqnarray}\label{E-3.22}
\nonumber
 3\,g(R_{\,\xi, Z}\,\varphi X, \varphi Y) = -4\,g(R_{\,\xi, Z}\,\widetilde Q\,\varphi X, \varphi Y) -4\,g(R_{\,\xi, Z}\,\varphi X, \widetilde Q\,\varphi Y) \\
 -4\,g(R_{\,\xi, Z}\,\widetilde Q\,\varphi X, \widetilde Q\,\varphi Y)
 + \delta(\varphi X, \varphi Y, Z, \,\xi) +2\,\delta(\varphi^2 X, \varphi^2 Y, Z, \,\xi).
\end{eqnarray}
From \eqref{E-3.22}, using \eqref{E-nS-04c}
and the equality $\delta(\varphi X,\varphi Y,Z,\xi)=0$,
we get \eqref{E-nS-04ccc}. Since $\varphi|_{\,\ker\eta}$ is non-degenerate, the distribution $\ker\eta$ is curvature invariant, see \eqref{E-nS-04cc}.
Using \eqref{E-3.18}, \eqref{E-3.23b} and symmetry of $h^2$ and $Q$, we obtain \eqref{E-3.6}.
Note that \eqref{E-3.23b} uses \eqref{E-3.24}--\eqref{E-3.23}, which require conditions \eqref{E-nS-10} and \eqref{E-nS-04c}.
Replacing $X$ by $\varphi X$ and $Y$ by $\varphi Y$ in \eqref{E-3.6} and using \eqref{E-nS-2.1}, we get \eqref{E-3.5}.
\hfill$\square$

\smallskip\textbf{Proof of Lemma~\ref{L-R03}}.
We prove \eqref{E-3.50} following the proof of (3.50) in~\cite{E-2005}.
Differentiating \eqref{E-3.29} and using
$g((\nabla_X\varphi)(\nabla_V\varphi)Y, Z)=-g((\nabla_V\varphi)Y, (\nabla_X\varphi)Z)$ gives
\begin{eqnarray}\label{E-3.32}
\nonumber
 && g((\nabla_V\,\varphi)Y, (\nabla_X\,\varphi)Z) + g((\nabla_X\,\varphi)Y, (\nabla_V\,\varphi)Z)
 = g((\nabla^2_{V,X}\,\varphi)\varphi Y, Z) \\
\nonumber
 && + g((\nabla^2_{V,X}\,\varphi)\varphi Z, Y)
    -g(hV,Z)\,g(hX,Y) -\eta(Z)\,g((\nabla_V h)X,Y) \\
 && - g(h V,Y)\,g(h X,Z) -\eta(Y)\,g((\nabla_V h)X,Z)
  -\nabla_V\big( \eta(Z)\,g(h X, \widetilde Q Y) \big).
\end{eqnarray}
Using \eqref{EF-nS-01}, \eqref{E-3.11} and \eqref{E-nS-2.1}, we find $\nabla^2$-terms in \eqref{E-3.32}:
\begin{eqnarray}
\label{ER-nS-04aa}
\nonumber
 && g((\nabla^2_{V,X}\,\varphi)\varphi Z, Y) = g((\nabla^2_{V, \varphi Z}\,\varphi)Y, X)
  = -g(R_{X,Y}V, Q Z) +\eta(Z)\,g(R_{X,Y}V, \xi) \\
 &&\qquad -(1/2)\,g(R_{X,V}\varphi Z, \varphi Y) - (1/2)\,g(R_{V,Y}\varphi Z, \varphi X) - (1/2)\,g(R_{X,Y}\varphi Z, \varphi V),\\
 \label{ER-nS-04a}
\nonumber
 && g((\nabla^2_{V,X}\,\varphi)\varphi Y, Z) = g((\nabla^2_{V, \varphi Y}\,\varphi)Z, X)
  = -g(R_{X,Z}V, Q Y) +\eta(Y)\,g(R_{X,Z}V, \xi) \\
 &&\qquad -(1/2)\,g(R_{X,V}\varphi Y, \varphi Z) - (1/2)\,g(R_{V,Z}\varphi Y, \varphi X) - (1/2)\,g(R_{X,Z}\varphi Y, \varphi V).
\end{eqnarray}
Using \eqref{E-3.23b}, we get from \eqref{E-3.32} and \eqref{ER-nS-04aa}--\eqref{ER-nS-04a} the equality
\begin{eqnarray}\label{ER-nS-03b}
\nonumber
 && g((\nabla_V\,\varphi)Y, (\nabla_X\,\varphi)Z) + g((\nabla_X\,\varphi)Y, (\nabla_V\,\varphi)Z) = -g(R_{X,Z}V, Q Y) \\
\nonumber
 && + \eta(Y)\,g(R_{X,Z}V, \xi) - g(R_{X,Y}V, Q Z) + \eta(Z)\,g(R_{X,Y}V, \xi) \\
\nonumber
 && - (1/2)\,g(R_{V,Z}\varphi Y, \varphi X) - (1/2)\,g(R_{X,Z}\varphi Y, \varphi V) - (1/2)\,g(R_{V,Y}\varphi Z, \varphi X)\\
\nonumber
 && - (1/2)\,g(R_{X,Y}\varphi Z, \varphi V) - g(hV,Z)\,g(hX,Y) - \eta(Z)\,g((\nabla_V h)X,Y) \\
 && - g(h V,Y)\,g(h X,Z) - \eta(Y)\,g((\nabla_V h)X,Z) - \nabla_V\big( \eta(Z)\,g(h X, \widetilde Q Y) \big).
\end{eqnarray}
From \eqref{ER-nS-03b}, applying \eqref{E-3.6} twice,
then replacing $(\nabla_V h)X$ by \eqref{E-3.24}, we get
\begin{eqnarray}\label{E-3.34}
\nonumber
 && g((\nabla_X\,\varphi)Z, (\nabla_V\,\varphi)Y) + g((\nabla_X\,\varphi)Y, (\nabla_V\,\varphi)Z) + g(R_{X,Z}V, Q Y) + g(R_{X,Y}V, Q Z)  \\
\nonumber
 && - g(R_{V,Z}\varphi X, \varphi Y) - g(R_{X,Z}\varphi V, \varphi Y) + g(h V,Y)\,g(h X,Z) + g(hV,Z)\,g(hX,Y) \\
 && = (1/4)\delta(X,Z,V,Y) - (1/4)\delta(V,Z,X,Y) -\nabla_V\big( \eta(Z)\,g(h X, \widetilde Q Y) \big).
\end{eqnarray}
Replacing in \eqref{E-3.34} $Z$ by $\varphi Z$ and $V$ by $\varphi V$, we find
\begin{align}\label{E-3.35}
\nonumber
 & g((\nabla_X\,\varphi)\varphi Z, (\nabla_{\varphi V}\,\varphi)Y) {+} g((\nabla_X\,\varphi)Y, (\nabla_{\varphi V}\,\varphi)\varphi Z)
 {+} g(R_{X,\varphi Z}\varphi V, Q Y) {-} g(R_{X,\varphi Z}\varphi^2 V, \varphi Y)   \\
\nonumber
 & + g(R_{X,Y}\varphi V,\varphi Q Z) - g(R_{\varphi V,\varphi Z}\varphi X,\varphi Y) + g(hX,Y)\,g(h\varphi V,\varphi Z) + g(h X,\varphi Z)\,g(h\varphi V,Y) \\
 & = (1/4)\,\delta(X,\varphi Z,\varphi V,Y) - (1/4)\,\delta(\varphi V,\varphi Z,X,Y) .
\end{align}
Using \eqref{E-3.29}, \eqref{E-3.30}, \eqref{E-3.31}, \eqref{E-nS-2.1} and Lemma~\ref{L-nS-02}, we find
\begin{eqnarray}\label{E-3.36}
\nonumber
&&\qquad g((\nabla_X\,\varphi)\varphi Z, (\nabla_{\varphi V}\,\varphi)Y) = g(Q(\nabla_X\,\varphi)Z, (\nabla_{V}\,\varphi)Y) \\
\nonumber
&& - g(X, \varphi h Z)\,g(V, \varphi h Y) - \eta(V)\,g((\nabla_X\,\varphi)Z, \varphi h Y) \\
&& + \eta(Z) g(\varphi h X, (\nabla_V\,\varphi)Y) - \eta(Z)\,\eta(V) g(h X, h Y) + g(h X, Q Z) g(h V, Q Y),\\
\label{E-3.37}
\nonumber
&&\qquad g((\nabla_X\,\varphi)Y, (\nabla_{\varphi V}\,\varphi)\varphi Z) = -g(Q(\nabla_X\,\varphi)Y, (\nabla_{V}\,\varphi)Z) \\
&& +\eta(V) g(\varphi h Z, (\nabla_{X} \varphi)Y) -\eta(Z) g(\varphi h V, (\nabla_{X} \varphi)Y)
 +g(\varphi h X, Y) g(\varphi h Z, \widetilde Q V).
\end{eqnarray}
From \eqref{E-3.23b} and Lemma~\ref{L-nS-02}, we have
\begin{align}\label{E-3.38}
\nonumber
 & g(R_{X, \varphi Z}\varphi V, Y) - g(R_{X, \varphi Z}\varphi^2 V, \varphi Y) = g(R_{X, \varphi Z}\varphi V, Y) \\
 & +g(R_{X, \varphi Z} Q V, \varphi Y) -\eta(X)\,\eta(V)\,g(h^2 Z, Q  Y).
\end{align}
On the other hand, from \eqref{E-3.4}, \eqref{E-3.6} and \eqref{E-nS-2.1} it follows that
\begin{eqnarray}\label{E-3.39}
 && g(R_{X, Z}\varphi V, \varphi Y) + g(R_{X, Z}\varphi^2 V, Y) + g(R_{X, \varphi Z}\varphi V, Y) + g(R_{\varphi X, Z}\varphi V, Y) = 0 ,\\
\label{E-3.40}
\nonumber
 && -\,g(R_{\varphi X,Z}\varphi V, Q Y) +\eta(Y)\,g(R_{\varphi X, Z}\varphi V,\xi) =-g(R_{Q X,\varphi Z}V,\varphi Y) \\
 &&\qquad +\,\eta(X)\,g(R_{\,\xi,\varphi Z}V,\varphi Y) +(1/2)\,\delta(\varphi X,Z,V,\varphi Y).
\end{eqnarray}
Summing up the formulas \eqref{E-3.39} and \eqref{E-3.40}
(and using
\eqref{E-3.40}, \eqref{E-3.23b} and Lemma~\ref{L-nS-02}), we obtain
\begin{eqnarray}\label{E-3.41}
\nonumber
 && g(R_{X, \varphi Z}\varphi V, Y) + g(R_{Q X,\varphi Z}V,\varphi Y) = \eta(X)\,\eta(V)\,g(h^2 Y, \widetilde Q Z) \\
\nonumber
 && -\eta(Y)\,\eta(Z)\,g(h^2 V, Q X) + g(R_{X, Z} Q V, Y)  - g(R_{X, Z}\varphi V, \varphi Y) \\
 && + \eta(Z)\,\eta(V)\,g(h^2 Y, X)
 + g(R_{\varphi X,Z}\varphi V, \widetilde Q Y) + (1/2)\,\delta(\varphi X,Z,V,\varphi Y).
\end{eqnarray}
Substituting \eqref{E-3.41} into \eqref{E-3.38}, we get
\begin{eqnarray}\label{E-3.42}
\nonumber
 && g(R_{X, \varphi Z}\,\varphi V, Y) - g(R_{X, \varphi Z}\,\varphi^2 V, \varphi Y) = g(R_{X, Z}\,Q V, Y) - g(R_{X, Z}\varphi V, \varphi Y) \\
\nonumber
 && + \eta(V)\,\eta(Z)\,g(h^2Y, X)  - \eta(Z)\,\eta(Y)\,g(h^2 X, Q V)  - \eta(X)\,\eta(V)\,g(h^2Z, Y) \\
 && - g(R_{\widetilde Q X, \varphi Z} V, \varphi Y) + g(R_{X, \varphi Z} \widetilde Q V, \varphi Y) + g(R_{\varphi X, Z}\varphi V, \widetilde Q Y) +(1/2)\,\delta(\varphi X,Z,V,\varphi Y) .
\end{eqnarray}
By means of \eqref{E-3.23b} and \eqref{E-3.6}--\eqref{E-3.5}, we have
\begin{eqnarray}\label{E-3.43}
\nonumber
 && g(R_{X, Y}\varphi V, \varphi Z) - g(R_{\varphi V, \varphi Z}\varphi X, \varphi Y) =  g(R_{V, Z}\varphi X,\varphi Y)  -g(R_{V, Z}Q X, Q Y) \\
\nonumber
 &&  -\eta(X)\,\eta(Z)\,g(h^2 Q Y, V) +\eta(X)\,\eta(V)\,g(h^2 Q Y, Z) +\eta(Y)\,\eta(Z)\,g(h^2 Q X, V) \\
 &&  -\eta(Y)\,\eta(V)\,g(h^2 Q X, Z) - (1/2)\,\delta(\varphi X,\varphi Y,Z,V) - (1/2)\,\delta(X,Y,Z,V) .
\end{eqnarray}
Substituting \eqref{E-3.36}, \eqref{E-3.37}, \eqref{E-3.42} and \eqref{E-3.43} into \eqref{E-3.35},
and using Lemma~\ref{L-nS-02}, we obtain
\begin{eqnarray}\label{E-3.44}
\nonumber
 && g(Q(\nabla_X\,\varphi)Z, (\nabla_{V}\,\varphi)Y) - g(Q(\nabla_X\,\varphi)Y, (\nabla_{V}\,\varphi)Z) \\
\nonumber
 &&  + \eta(V) g((\nabla_{X} \varphi)Y, \varphi h Z) - \eta(V)\,g((\nabla_X\,\varphi)Z, \varphi h Y)
 +\eta(Z) g((\nabla_V\,\varphi)Y, \varphi h X) \\
\nonumber
 && - \eta(Z) g((\nabla_{X} \varphi)Y, \varphi h V) + g(h X, Q Z) g(h V, Q Y) - g(hX,Y)\,g(h V, Q Z) \\
\nonumber
 && + 2\,\eta(Z)\,\eta(V) g(h^2 X, Y) - \eta(X)\,\eta(Z)\,g(h^2 Y, QV) - \eta(Y)\,\eta(V)\,g(h^2 X, QZ) \\
\nonumber
 && + g(R_{X, Z}\,Q V, Y) - g(R_{X, Z}\varphi V, \varphi Y) + g(R_{V, Z}\varphi X,\varphi Y) - g(R_{V, Z}\,Q X, Q Y)
\end{eqnarray}
\begin{eqnarray}
\nonumber
&& = g(R_{\widetilde Q X, \varphi Z} V, \varphi Y) {-} g(R_{X, \varphi Z} \widetilde Q V, \varphi Y) {-} g(R_{X,\varphi Z}\varphi V, \widetilde Q Y)
{-} g(R_{X,Y}\varphi V, \varphi \widetilde Q Z) \\
\nonumber
 && - g(R_{\varphi X, Z}\varphi V, \widetilde Q Y) - \eta(X)\,\eta(V)\,g(h^2 Y, \widetilde QZ) - g(\varphi h X, Y) g(\varphi h Z, \widetilde Q V) \\
 \nonumber
 && - (1/2)\,\delta(\varphi X,Z,V,\varphi Y) + (1/2)\,\delta(\varphi X,\varphi Y,Z,V) + (1/2)\,\delta(X,Y,Z,V) \\
 && +(1/4)\,\delta(X,\varphi Z,\varphi V,Y) - (1/4)\,\delta(\varphi V,\varphi Z,X,Y) . 
\end{eqnarray}
Adding \eqref{E-3.44} to \eqref{E-3.34}, we obtain
\begin{eqnarray}\label{E-3.45}
\nonumber
 && 2\,g((\nabla_X\,\varphi)Z, (\nabla_V\,\varphi)Y) + g(\widetilde Q(\nabla_X\,\varphi)Z, (\nabla_{V}\,\varphi)Y)
 - g(\widetilde Q(\nabla_X\,\varphi)Y, (\nabla_{V}\,\varphi)Z) \\
\nonumber
 && + \eta(V) g((\nabla_{X} \varphi)Y, \varphi h Z) - \eta(V)\,g((\nabla_X\,\varphi)Z, \varphi h Y) +\eta(Z) g((\nabla_V\,\varphi)Y, \varphi h X) \\
\nonumber
 && - \eta(Z) g((\nabla_{X} \varphi)Y, \varphi h V) + g(h X, Q Z) g(h V, Q Y) + g(h V,Y)\,g(h X,Z) \\
\nonumber
 && + 2\,\eta(Z)\,\eta(V) g(h^2 X, Y) - \eta(X)\,\eta(Z)\,g(h^2 Y, QV) - \eta(Y)\,\eta(V)\,g(h^2 X, QZ) \\
\nonumber
 && + 2\,g(R_{X, Z}\,V, Y) - 2\,g(R_{X,Z}\varphi V, \varphi Y) - g(R_{V, Z}\,Q X, Q Y) + g(R_{X,Y}V, Q Z) \\
\nonumber
&& = g(R_{\widetilde Q X, \varphi Z} V, \varphi Y) {-} g(R_{X, \varphi Z} \widetilde Q V, \varphi Y) {-} g(R_{X,\varphi Z}\varphi V, \widetilde Q Y) \\
\nonumber
 && - g(R_{X, Z}\,\widetilde Q V, Y) -g(R_{X,Z}V, \widetilde Q Y) -g(R_{X,Y}\varphi V, \varphi\widetilde Q Z) -g(R_{\varphi X, Z}\varphi V,\widetilde Q Y) \\
\nonumber
 && - \eta(X)\,\eta(V)\,g(h^2 Y, \widetilde QZ) + g(hX,Y)\,g(h V, \widetilde Q Z) - g(\varphi h X, Y) g(\varphi h Z, \widetilde Q V) \\
 \nonumber
 && - (1/2)\,\delta(\varphi X,Z,V,\varphi Y) + (1/2)\,\delta(\varphi X,\varphi Y,Z,V) + (1/2)\,\delta(X,Y,Z,V) \\
 \nonumber
 && +(1/4)\delta(X,Z,V,Y) - (1/4)\delta(V,Z,X,Y) -\nabla_V\big( \eta(Z)\,g(h X, \widetilde Q Y) \big) \\
 && +(1/4)\,\delta(X,\varphi Z,\varphi V,Y) - (1/4)\,\delta(\varphi V,\varphi Z,X,Y).
\end{eqnarray}
Swapping $X\leftrightarrow Z$ and $V\leftrightarrow Y$ in \eqref{E-3.45},
then subtracting the gotten equation from \eqref{E-3.45} and using \eqref{E-nS-00b}, we~get
\begin{eqnarray}\label{E-3.46}
\nonumber
 && \eta(V) g((\nabla_{X} \varphi)Y, \varphi h Z) - \eta(V)\,g((\nabla_X\,\varphi)Z, \varphi h Y) +\eta(Z) g((\nabla_V\,\varphi)Y, \varphi h X) \\
\nonumber
 && - \eta(Z) g((\nabla_{X} \varphi)Y, \varphi h V) - \eta(Y) g((\nabla_{Z} \varphi)V, \varphi h X) + \eta(Y)\,g((\nabla_Z\,\varphi)X, \varphi h V) \\
\nonumber
 &&  -\eta(X) g((\nabla_Y\,\varphi)V, \varphi h Z) + \eta(X) g((\nabla_{Z} \varphi)V, \varphi h Y) \\
\nonumber
 && + 2\,\eta(Z)\,\eta(V) g(h^2 X, Y) - 2\,\eta(X)\,\eta(Y) g(h^2 Z, V)  \\
\nonumber
 &&  + g(R_{Z,V}\,Q X, Q Y) + g(R_{Z,V}\,Q X, Y) - g(R_{X,Y}\,Q Z, V) - g(R_{X,Y}\,Q Z, Q V) \\
\nonumber
&& = g(R_{\widetilde Q X, \varphi Z} V, \varphi Y) - g(R_{X,\varphi Z}\varphi V, \widetilde Q Y) + g(R_{Z,V}\varphi Y, \varphi\widetilde Q X) \\
 \nonumber
 &&  - g(R_{X,Y}\varphi V, \varphi\widetilde Q Z) - g(R_{\widetilde Q Z, \varphi X} Y, \varphi V) + g(R_{Z,\varphi X}\varphi Y, \widetilde Q V) \\
\nonumber
 && - \eta(X)\,\eta(V)\,g(h^2 Y, \widetilde QZ) + g(hX,Y)\,g(h V, \widetilde Q Z) - g(\varphi h X, Y) g(\varphi h Z, \widetilde Q V) \\
\nonumber
 && + \eta(Z)\,\eta(Y)\,g(h^2 V, \widetilde QX) - g(hZ,V)\,g(h Y, \widetilde Q X) + g(\varphi h Z, V) g(\varphi h X, \widetilde Q Y) \\
 \nonumber
 && +(1/4)\,\delta(\varphi X,Z,\varphi Y,V) + (1/4)\,\delta(\varphi X,\varphi Y,Z,V) + (1/2)\delta(X,Y,Z,V) \\
 \nonumber
 && + (1/4)\,\delta(X,\varphi Z,\varphi V,Y) - (1/4)\,\delta(\varphi Z,\varphi V, X,Y) \\
 && + \nabla_Y\big( \eta(X)\,g(h Z, \widetilde Q V) \big)  - \nabla_V\big( \eta(Z)\,g(h X, \widetilde Q Y) \big) .
\end{eqnarray}
Putting $\xi$ on $V$ of \eqref{E-3.46},
then using Lemma~\ref{L-nS-02}, \eqref{E-nS-04cc},
$\nabla_\xi\big( \eta(Z)\,g(h X, \widetilde Q Y) \big)=0$, Remark~\ref{Rem-delta} (that all $\delta$-terms vanish),
\eqref{E-3.23b} and \eqref{E-nS-00b}, we get
\begin{eqnarray}\label{E-3.48}
\nonumber
 && g((\nabla_{Y} \varphi)X, \varphi h Z) + g((\nabla_X\,\varphi)Z, \varphi h Y) \\
 && = \eta(X)\,g(a_1(Z),Y) +\eta(Y)\,g(a_2(Z),X)+\eta(Z)\,g(a_3(X),Y) ,
\end{eqnarray}
where $a_1,a_2,a_3$ are self-adjoint linear operators on $TM$.
Swapping $X$ and $Y$ in \eqref{E-3.48}
and substituting \eqref{E-3.48} into the resulting equation, we get
\begin{eqnarray}\label{E-3.49}
\nonumber
 && g((\nabla_Z\,\varphi)X, \varphi h Y) + \,g((\nabla_Z\,\varphi)Y, \varphi h X) \\
 && = \eta(X)\,g(b_1(Z),Y) +\eta(Y)\,g(b_2(Z),X)+\eta(Z)\,g(b_3(X),Y) ,
\end{eqnarray}
where $b_1,b_2,b_3$ are self-adjoint linear operators on $TM$.
By swapping $Z$ and $X$ in \eqref{E-3.49},
and adding the resulting equation and \eqref{E-3.48} we get
\begin{eqnarray}\label{E-3.48b}
 g((\nabla_X\,\varphi)Y, \varphi h Z) = \eta(X)\,g(c_1(Z),Y) +\eta(Y)\,g(c_2(Z),X)+\eta(Z)\,g(c_3(X),Y) ,
\end{eqnarray}
where $c_1,c_2,c_3$ are self-adjoint linear operators on $TM$ and $c_1,c_2$ vanish on $\xi$.
Taking $Z=\xi$ in \eqref{E-3.48b}, we get $c_3=0$.
Taking $X=\xi$ in \eqref{E-3.48b} and using $(\nabla_\xi\,\varphi)X = \varphi h X$,
we find $c_1 = \varphi^2 h^2$.
Taking $Y=\xi$ in \eqref{E-3.48b}, and using \eqref{E-nS-01c}, we find $c_2 = -\varphi^2 h^2$.
Thus \eqref{E-3.48b} reduces to \eqref{E-3.50}.
\hfill$\square$

\bigskip

\textbf{Conclusions}.
We have shown that the weak nearly cosymplectic structure is useful for studying almost contact metric structures and Killing vector fields.
Some results on nearly cosymplectic manifolds (see \cite{C-MD-2016,E-2005,NDY-2018}) were extended
to weak nearly cosymplectic manifolds satisfying \eqref{E-nS-10} and \eqref{E-nS-04c} and the splitting theorem was proven.
Our conjecture is that the conditions \eqref{E-nS-10} and \eqref{E-nS-04c} are also sufficient for a weak nearly Sasakian manifold of dimension greater than five to be Sasakian -- this could answer the question in \cite{rov-2023} and generalize Theorem~3.3 in \cite{NDY-2018}.
Based on the numerous applications of nearly cosymplectic structures, we expect that certain weak structures will also be useful for differential geometry and physics, for example, in twistor string theory and~QFT.

\end{document}